\documentclass[11pt]{article}
\usepackage[T1]{fontenc}
\usepackage[margin=1in]{geometry}
\usepackage{amsmath,amssymb,amsthm,mathtools,bm}
\usepackage{booktabs,tabularx,array,graphicx,microtype,enumitem,xcolor,placeins,float}
\usepackage[font=small,labelfont=bf]{caption}
\usepackage[round,authoryear]{natbib}
\usepackage[colorlinks=true,allcolors=blue!55!black]{hyperref}
\hypersetup{
  pdftitle={Same Resident Strains, Different Attractors: Opposite Local Growth Signs for a Rare Third Strain},
  pdfauthor={Ruiwu Niu, Xincheng Shu, Ying Zhao, Jingyi Wang},
  pdfsubject={Multistrain epidemic dynamics, resident bistability, and invasion from rarity}
}

\graphicspath{{figures/}}
\setlist{nosep,leftmargin=*}
\newtheorem{proposition}{Proposition}[section]
\newtheorem{corollary}[proposition]{Corollary}
\theoremstyle{definition}
\newtheorem{computationalresult}[proposition]{Computational Result}
\newcommand{\R}{\mathbb{R}}
\newcommand{\one}{\bm 1}
\newcommand{\dd}{\mathrm{d}}
\newcommand{\e}{\mathrm{e}}
\newcommand{\cF}{\mathcal F}
\newcommand{\cM}{\mathcal M}
\newcommand{\cK}{\mathcal K}
\newcommand{\D}{\Delta}
\newcommand{\sbound}{\operatorname{s}}

\title{Same Resident Strains, Different Attractors:\\Opposite Local Growth Signs for a Rare Third Strain}
\author{Ruiwu Niu\textsuperscript{1},
Xincheng Shu\textsuperscript{2,3},
Ying Zhao\textsuperscript{4}, and
Jingyi Wang\textsuperscript{5}\\[0.35em]
\parbox{0.92\textwidth}{\centering\small
\textsuperscript{1}Department of Applied Data Science, Hong Kong Shue Yan University,
Hong Kong, Hong Kong SAR, China\\
\textsuperscript{2}Computational Communication Research Center, Beijing Normal University,
Zhuhai 519087, China\\
\textsuperscript{3}School of Journalism and Communication, Beijing Normal University,
Beijing 100875, China\\
\textsuperscript{4}Department of Electrical Engineering, City University of Hong Kong,
Kowloon, Hong Kong SAR, China\\
\textsuperscript{5}School of Mathematical Sciences, Shenzhen University,
Shenzhen 518060, PR China\\[0.2em]
\href{mailto:rniu@hksyu.edu}{\texttt{rniu@hksyu.edu}};
\href{mailto:yzhao396-c@my.cityu.edu.hk}{\texttt{yzhao396-c@my.cityu.edu.hk}};
\href{mailto:wangjingyi@szu.edu.cn}{\texttt{wangjingyi@szu.edu.cn}}}}
\date{}

\begin{document}
\maketitle
\begin{abstract}
Most models assess whether a newly introduced pathogen strain can grow when
rare by testing it against a steady endemic background. Resident strains,
however, need not have a single long-term behavior. Under the same parameter
values they may settle to a steady state or continue through recurring
outbreaks, and these outcomes can leave different fractions of hosts with
different infection histories. We study this possibility in a stylized
model that records immune history. For two resident strains, we find a parameter point at
which a stable endemic equilibrium and a stable outbreak cycle are both
locally attracting. Numerical continuation maps a Hopf branch and a
branch of folds of periodic orbits that bound a candidate bistability region;
their local geometry is organized by a generalized Hopf (Bautin) point. We
then embed the resident model exactly in a system with 20 states and three strains. In a
benchmark for which strain 3 infectiousness and total removal are the same
across the four host history classes, whether strain 3 grows or declines when
rare is determined by a weighted sum of the corresponding uninfected host
fractions. The equilibrium and cycle produce sufficiently different weighted
sums that strain 3, with one fixed parameter set, grows near one resident
attractor but declines near the other. A second parameter set
gives the reverse ordering, and both patterns persist over open regions of
invader parameter space. Sixty simulations of the full model with
progressively smaller inocula reproduce these linear
predictions across cycle phases and numerical solvers. The central conclusion
is that knowing which resident strains are present may not be enough to predict
initial invasion: which long-term resident attractor is realized can also matter.
Our results concern local growth from rarity, not
eventual establishment, long-term coexistence, or the dynamics after invasion.
\end{abstract}

\noindent\textbf{Keywords:} multistrain epidemic; immune history;
susceptibility shaped by infection history; susceptibility enhancement; Bautin bifurcation;
fold of periodic orbits; Floquet multiplier; alternative attractors

\section{Introduction}

The first question asked about a new pathogen strain is often whether it can
increase when rare. Classical next-generation calculations formalize this
question at a disease-free equilibrium \citep{VanDenDriesscheWatmough2002}.
Related threshold analyses have considered deterministic SEIHR models with
human migration, stochastic SEIHR systems, and SIR dynamics on metapopulation
networks coupled by interregional migration
\citep{NiuEtAl2020SEIHRMigration,NiuEtAl2021StochasticSEIHR,
NiuEtAl2024NetworkMigration}. Multistrain studies
use the same logic at resident endemic equilibria
\citep{AdamsSasaki2007}. An equilibrium is an appropriate reference when the
resident endemic equilibrium is the relevant long-term outcome. It is
incomplete when the same
resident strains, under the same parameter values, can approach different
long-term states.

Here a \emph{resident} is the strain community already circulating before a
third strain is introduced. An \emph{attractor} is a compact invariant set that
attracts trajectories from a neighborhood. The two attractors in our example are a
steady endemic equilibrium and a sustained outbreak cycle. They contain the
same two resident strains, but they do not expose the rare third strain to the
same host population. Recovery leaves hosts with a record of earlier
infections; this \emph{immune history} changes their later susceptibility.
Consequently, two resident attractors can generate different long-run
distributions of host histories even when the strains present and all model
parameters are identical.

Multistrain epidemic models that record immune history have a long lineage. Early
cross-immunity formulations led to descriptions based on exposure sets and infection status,
invasion criteria, endemic multistrain states, and oscillatory dynamics
\citep{CastilloChavezEtAl1989,AndreasenLinLevin1997,GogSwinton2002,DawesGog2002}.
Antigenically structured models can also sustain complex periodic or chaotic
dynamics \citep{GuptaFergusonAnderson1998}, and later work developed systematic
reductions for pathogens with many strains \citep{KucharskiAndreasenGog2016}
and interactions among immune histories at multiple loci \citep{McLeodBankGandon2024}.
Seasonally forced cross-immunity models already display multiple attractors and
complex cycles \citep{KamoSasaki2002}. Recent eco-evolutionary work likewise
shows that an invading variant can reshape the immune environment that first
favored it \citep{BarratCharlaixNeher2024}. State spaces that record immune history,
oscillations driven by cross-immunity, and feedback between variants and host
immunity are therefore established ingredients. Our focus is a specific
consequence that arises when alternative resident attractors are both stable.

The resident model used here has especially close antecedents.
\citet{ChungLui2016} analyzed an antecedent model with eight compartments
and partial cross-immunity. More recently, \citet{Gavish2024Oscillatory}
studied the corresponding subfamily without waning and with partial cross-immunity in this
eight-state structure, analyzed the loss of stability of its coexistence
equilibrium, and exhibited periodic solutions numerically; a related extension
addresses stable coexistence under large competitive advantage
\citep{Gavish2025Exclusion}. Generalized Hopf bifurcations, folds of periodic
orbits, and bistability between an equilibrium and a periodic orbit also have
precedents in epidemic models
\citep{MoghadasAlexander2006,OpokuSarkodieEtAl2024,ScarabelEtAl2025,
ScarabelColdwellCassidy2026}. Our
question is not whether any one of these phenomena can occur. We ask whether
their two-parameter bifurcation structure creates alternative resident attractors with
different consequences for a rare third strain.

Invasion of an endemic community containing two resident strains by a third strain was studied directly
by \citet{AdamsSasaki2007}. More generally, invasion growth rates indexed by boundary
invariant measures and attractor inheritance are established ideas in
population dynamics \citep{Schreiber2000,GeritzEtAl2002,CaiGeritz2020};
Floquet invasion criteria and criteria based on time averages for periodic epidemic environments
are likewise available \citep{MitchellKribs2019}. In the theory of invasion graphs,
the sign parity condition requires ergodic invariant measures
that have the same support, which here means the same set of resident strains,
to give a missing strain the same invasion sign
\citep{HofbauerSchreiber2022}. The theory also recognizes that this assumption
can fail outside special model classes. Our contribution is a constructive
failure of sign parity among measures with the same support in an immune history epidemic model: the ergodic
invariant measures supported on two stable resident attractors in the same
invariant face with two resident strains give opposite signs to the same fixed invader. A
local structure formed by the Bautin point and LPC curve underlies this
computed bistability between the equilibrium and the cycle.

Recent work separates initial invasion from long-run co-circulation and adapts
ecological coexistence theory to strain competition
\citep{ParkEtAl2024,ParkLevineGrenfell2026}.
We keep the same distinction throughout. ``Invasion'' below means positive
local growth from rarity near a specified resident attractor. It does not, by
itself, imply establishment, persistence, or stable three-strain coexistence.

The study has three linked parts. First, we identify a parameter region in
which the resident two-strain system has both a stable equilibrium and a stable
cycle. Second, we embed that system exactly in a three-strain model and derive
the third strain's local growth rate around each resident attractor. Third, we
show that holding the resident parameters and the invader parameter set fixed
while changing only the resident attractor reverses the growth sign. With a
second invader parameter set, strain 3 has the reverse ordering, and both patterns persist
over open regions of
invader parameter space. Figure~\ref{fig:story} gives the logic in one view. At the
same resident parameter point, different initial conditions can lead to a
steady endemic state or a recurring epidemic cycle. Those outcomes leave
different mixtures of host immune histories, so a rare third strain does not
encounter the same environment. Its initial growth can therefore change sign
even though the resident strains, resident parameters, and parameter set for the third strain
are unchanged.

\FloatBarrier

\begin{figure}[t]
\centering
\includegraphics[width=\textwidth]{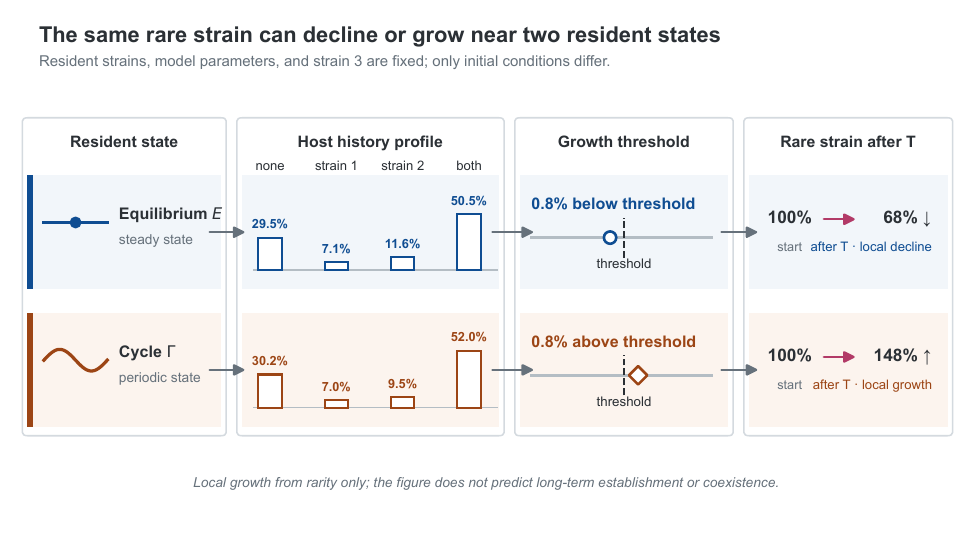}
\caption{How two resident states give the same rare strain different local
growth conditions. At \(p_0=(\gamma_2,\eta)=(1,10^{-3})\), both rows keep the
resident strains, resident parameters, and complete parameter set of strain 3
fixed. Different initial conditions select either the equilibrium
\(E\) in the top row or the cycle \(\Gamma\) in the bottom row. Both are stable
attractors of the resident subsystem without strain 3. The four
small bars show the fractions of the full population in the uninfected history
classes with no prior infection, prior strain 1, prior strain 2, or both
histories. They are equilibrium values for \(E\) and time averages over one
cycle for \(\Gamma\), and all bars use the same zero baseline and scale. For
the fixed strain 3, the effective transmission multipliers associated with
these four classes are \(1\), \(0.2\), \(0.1\), and \(0.02\). Applying these
weights gives \(D_E=0.3308\) and
\(D_\Gamma=0.3358\). The fixed strain 3 has growth threshold
\(D^*=0.3333\), so its effective opportunity is 0.8\% below the threshold near
\(E\) and 0.8\% above it near \(\Gamma\). Over the same interval \(T\), equal
to one period of the resident cycle, rare active strain 3 infection is
multiplied by 0.68 near \(E\) and by 1.48 near \(\Gamma\). The complete history
values for these four uninfected history classes and their ranges over the
cycle are shown in Fig.~\ref{fig:invasion}(a). In this benchmark,
infectiousness and total removal do not depend on history, so
Eq.~\eqref{eq:measure_r} makes the weighted score exact. In a general periodic
environment the corresponding criterion is obtained from the Floquet
monodromy matrix. These factors describe local growth from rarity and do not
predict establishment, coexistence, or exclusion over the long term after a
finite introduction.}
\label{fig:story}
\end{figure}

\section{The resident subsystem with two strains}
\label{sec:resident}
Before strain 3 is introduced, the dynamics are governed by the subsystem with
two strains described below. We call this restriction without strain 3 the
\emph{resident subsystem with two strains}, following the terminology for residents
and invaders used by \citet{CaiGeritz2020}. It is the
form without mutation, normalized by total population, of the competitive SIRS
model with two variants introduced by \citet{NiuEtAl2023}. This section restates the model in the
notation for immune history used for the extension to three strains below.
\FloatBarrier

\subsection{Immune history equations for eight states}

All state variables are population fractions. Let
\[
x=(S,I_1,I_2,I_{1\mid2},I_{2\mid1},R_1,R_2,P_{12})^\top.
\]
Here \(S\) denotes hosts with no recorded immunity, \(I_i\) primary infection
with strain \(i\), \(I_{1\mid2}\) infection by strain 1 after strain 2, and
\(I_{2\mid1}\) the reverse order. The uninfected states \(R_i\) retain only
history \(i\), whereas \(P_{12}\) retains both histories. Put
\[
Y_1=I_1+I_{1\mid2},\qquad Y_2=I_2+I_{2\mid1}.
\]
Primary effective transmission coefficients are \(\alpha_i\), and
history-specific secondary effective transmission coefficients are \(\gamma_i\).
The recovery rates are \(\beta_i\) and
\(\beta_{i\mid j}\). Memory is lost at rate \(\eta\); balanced birth and
background death occur at rate \(d\). Mutation is excluded.

The host flow is simple even though the notation is compact. A host with no
recorded history begins in \(S\). Infection with strain \(i\) moves the host
to \(I_i\), and recovery moves the host to \(R_i\). A later infection with the
other strain moves the host to \(I_{j\mid i}\), after which recovery adds both
infections to the recorded history \(P_{12}\). While uninfected, a host loses
the entire recorded history at rate \(\eta\) and returns to \(S\). Thus
``immune history'' is a coarse model state recording which strains have been
cleared and, during infection, the current strain. It is neither a measured
immunological titre nor a complete chronological exposure record.

The resident subsystem is
\begin{align}
\dot S={}&-(\alpha_1Y_1+\alpha_2Y_2+d)S+\eta(R_1+R_2+P_{12})+d,\label{eq:residentS}\\
\dot I_1={}&\alpha_1Y_1S-(\beta_1+d)I_1,\\
\dot I_2={}&\alpha_2Y_2S-(\beta_2+d)I_2,\\
\dot I_{1\mid2}={}&\gamma_1Y_1R_2-(\beta_{1\mid2}+d)I_{1\mid2},\\
\dot I_{2\mid1}={}&\gamma_2Y_2R_1-(\beta_{2\mid1}+d)I_{2\mid1},\\
\dot R_1={}&\beta_1I_1-(\gamma_2Y_2+\eta+d)R_1,\\
\dot R_2={}&\beta_2I_2-(\gamma_1Y_1+\eta+d)R_2,\\
\dot P_{12}={}&\beta_{1\mid2}I_{1\mid2}+\beta_{2\mid1}I_{2\mid1}
-(\eta+d)P_{12}.\label{eq:residentP}
\end{align}
The bifurcation parameters are the strain-2 secondary effective transmission coefficient
\(\gamma_2\) and the waning rate \(\eta\). Table~\ref{tab:parameters} lists
the fixed benchmark values.

\begin{table}[H]
\centering
\caption{Resident benchmark. Rates are per day. The two entries in the last
row vary in the two-parameter computations.}
\label{tab:parameters}
\begin{tabular}{lll}
\toprule
Parameter & Value & Interpretation\\
\midrule
\(\alpha_1,\alpha_2\) & \(0.3,\,0.1\) & transmission with no recorded history\\
\(\beta_1,\beta_2,\beta_{1\mid2},\beta_{2\mid1}\) & \(0.1\) & recovery\\
\(\gamma_1\) & \(0.1\) & strain-1 effective transmission coefficient after strain 2\\
\(d\) & \(10^{-4}\) & balanced birth/background death\\
\(\gamma_2,\eta\) & variable & strain-2 effective transmission coefficient after strain 1; waning\\
\bottomrule
\end{tabular}
\end{table}

The benchmark is a stylized example for the analysis of nonlinear dynamics rather than a
calibration to a particular pathogen. At the parameter point used in the main
example,
\(\gamma_2=1\) whereas \(\alpha_2=0.1\): a recorded strain-1 infection
increases the modeled effective transmission coefficient for strain 2 tenfold.
The model therefore allows both cross-protection and phenomenological
susceptibility enhancement associated with infection history; in this benchmark,
prior strain 1 increases susceptibility to strain 2. Complete memory reset
is another simplifying assumption. These choices provide a clean setting in
which to isolate the attractor effect; they should not be read as fitted
biological estimates.

\FloatBarrier

\begin{proposition}[Positive simplex and tangent-space linearization]
\label{prop:simplex8}
For nonnegative finite rates and \(d>0\), the simplex
\[
\D_7=\{x\in\R_+^8:\one^\top x=1\}
\]
is positively invariant under the system from \eqref{eq:residentS} through
\eqref{eq:residentP}, and
solutions starting in \(\D_7\) exist for all forward time. If
\(Q_\D\in\R^{8\times7}\) has orthonormal columns spanning \(\one^\perp\), the
tangent-space linearization is
\[
J_\D=Q_\D^\top D_xf(x)Q_\D.
\]
Hopf and Floquet stability on \(\D_7\) must be determined from this tangent
dynamics, not from the direction normal to the constant-population simplex.
\end{proposition}
\begin{proof}
Every coordinate derivative is nonnegative when that coordinate is zero.
Direct summation gives
\[
\frac{\dd}{\dd t}\one^\top x=d(1-\one^\top x).
\]
Unit total population is therefore preserved. Compactness of the simplex and
local Lipschitz continuity give global forward existence. The tangent-space
formula restricts the Jacobian to the invariant affine hyperplane.
\end{proof}

\subsection{Conditions for Hopf, generalized Hopf, and folds of periodic orbits on the simplex}

At a coexistence equilibrium \(x^*(p)\), a Hopf point of the
simplex-restricted dynamics has one simple
pair \(\pm i\omega\) in \(J_\D\), \(\omega>0\), no other spectrum on the
imaginary axis, and a transverse crossing along the equilibrium branch. In
local tangent coordinates, define
\[
g(z;p)=Q_\D^\top f\bigl(x^*(p)+Q_\D z;p\bigr).
\]
At the Hopf point we use the factorial normal form convention
\[
g(z;p^*)=J_\D z+\tfrac12B(z,z)+\tfrac16C(z,z,z)+\cdots,
\]
with \(J_\D q=i\omega q\), \(J_\D^H\ell=-i\omega \ell\), and \(\ell^Hq=1\).
The first and second Lyapunov coefficients are denoted \(l_1\) and \(l_2\).
A nondegenerate generalized Hopf, or Bautin, point additionally satisfies
\(l_1=0\), nonzero variation of \(l_1\) along the Hopf curve, and \(l_2\ne0\).

For a periodic boundary-value problem, a phase condition fixes the
time-translation degeneracy. The full resident monodromy matrix still has the
Floquet multiplier \(+1\) associated with time translation. At a
nondegenerate \emph{limit point of cycles} (LPC), equivalently a fold
(saddle-node bifurcation) of periodic orbits, there is a second multiplier
\(+1\); it is simple after restriction to a Poincar\'e section. The
fixed-parameter boundary-value problem has a one-dimensional kernel and
satisfies the standard parameter-transversality and quadratic nondegeneracy
conditions.

These terms have a direct dynamical meaning. A Hopf bifurcation is where an
equilibrium gains or loses an oscillatory mode. A generalized Hopf point is
where the Hopf bifurcation changes from producing a stable small-amplitude
periodic orbit to an unstable one, or conversely. At the LPC considered here,
a stable and an unstable periodic orbit meet while the remaining nontrivial
multipliers lie inside the unit circle. Their local organization can therefore
create a region in which the equilibrium and cycle are both stable.

\begin{proposition}[Truncated Bautin normal form and local organization]
\label{prop:bautinlocal}
Suppose, in addition to the nondegenerate Bautin conditions above, that every
noncritical eigenvalue has negative real part. Near such a point with
\(l_2<0\), smooth state, time, and parameter changes give the quintic
truncated radial normal form
\[
\dot\varrho=\varrho\{\xi_1+\xi_2\varrho^2-\varrho^4\}.
\]
For \(\xi_2>0\) and \(-\xi_2^2/4<\xi_1<0\), this truncation contains a stable
equilibrium, an inner unstable cycle, and an outer stable cycle. Its two cycles
meet on \(\xi_1=-\xi_2^2/4\). In the full smooth system, the corresponding
local branches and their stability persist sufficiently near the Bautin point
away from the bifurcation curves, and the LPC satisfies
\(\xi_1=-\xi_2^2/4+o(\xi_2^2)\).
\end{proposition}
\begin{proof}
This is the standard two-parameter generalized Hopf normal form
\citep{Kuznetsov2023}. For the quintic truncation, writing
\(r=\varrho^2\) gives the two positive roots of
\(\xi_1+\xi_2r-r^2=0\); their radial derivatives give the stated stability
ordering. Persistence and the LPC asymptotic follow from the nondegenerate
generalized Hopf unfolding.
\end{proof}

Proposition~\ref{prop:bautinlocal} is a conditional local theorem. The
coordinates and branches reported below are numerical rather than analytic
existence results.

\section{Embedding with three strains and invasion determined by the attractor}
\label{sec:invasion_theory}
\subsection{An immune history model with 20 states}

The three-strain construction applies the same bookkeeping rule to all subsets
of \(\{1,2,3\}\). An uninfected class records which strains a host has previously
cleared. An infected class records that history together with the current
strain. There are eight possible uninfected histories and twelve allowable
infection states, giving 20 states in total. Homologous reinfection is excluded:
a host whose history already contains strain \(i\) cannot acquire strain \(i\) again.

Figure~\ref{fig:state_model} shows how one infection and recovery rule generates
all 20 compartments. It is a map of compartment labels and transitions, rather
than a low-dimensional phase portrait.

\begin{figure}[t]
\centering
\includegraphics[width=\textwidth]{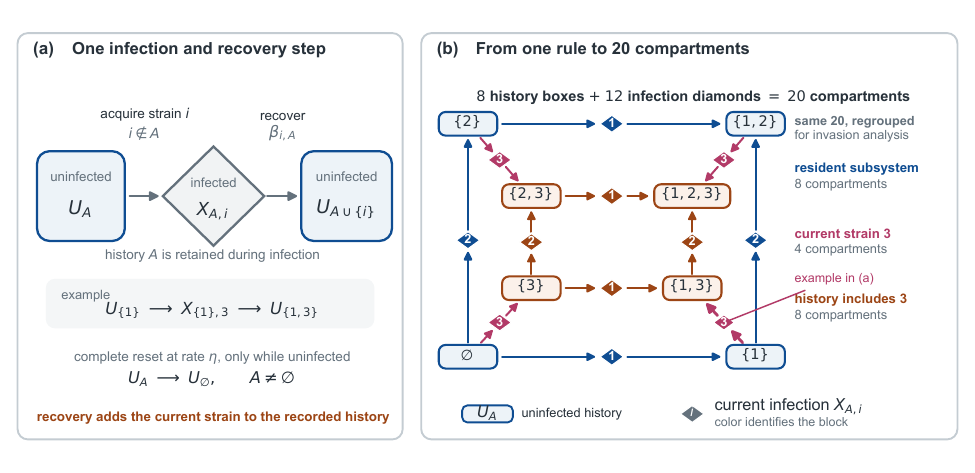}
\caption{Construction of the immune history model with three strains. (a) An
uninfected compartment \(U_A\) records the set \(A\) of strains cleared
previously. If \(i\notin A\), infection moves a host to \(X_{A,i}\) while
retaining \(A\); recovery adds \(i\) and moves the host to
\(U_{A\cup\{i\}}\). Complete reset \(U_A\to U_\varnothing\) occurs at rate
\(\eta\) for \(A\ne\varnothing\), only while the host is uninfected. Births enter \(U_\varnothing\),
and background death acts on every compartment. (b) The eight possible
uninfected histories form the vertices of a Boolean cube, where moving along
an edge adds one strain to the recorded history. Each of its 12 directed edges
is split by one current infection compartment, giving \(8+12=20\)
compartments. Rounded boxes denote \(U_A\), and diamonds marked \(i\) denote
\(X_{A,i}\). The emphasized path
\(U_{\{1\}}\to X_{\{1\},3}\to U_{\{1,3\}}\) is the example from panel (a).
Colors group compartments for the invasion analysis; they do not assign a
general color to each strain. The outer square and its strain 1 or 2 edge states give the
resident subsystem with eight states. The four connecting diamonds form the
active strain 3 block \(z_3\). The inner square contains the eight states in
\(h_3\) whose history includes strain 3. This is a compartment indexing
diagram, not a phase portrait. Homologous reinfection, simultaneous
coinfection, and mutation are excluded.}
\label{fig:state_model}
\end{figure}
\FloatBarrier

Let \(\mathcal N=\{1,2,3\}\). For each \(A\subseteq\mathcal N\), let \(U_A\)
be the uninfected fraction with immune history \(A\). For \(i\notin A\), let
\(X_{A,i}\) be the fraction infected by strain \(i\) with pre-infection history
\(A\). Thus the state count is \(2^3+3\,2^2=20\).
Define
\[
\Lambda_i=\sum_{A:i\notin A}\tau_{i,A}X_{A,i},
\]
where \(c_{i,A}\ge0\) is the history-specific effective transmission
coefficient for strain \(i\) in hosts with history \(A\),
\(\tau_{i,A}\ge0\) is relative
infectiousness, and \(\beta_{i,A}\ge0\) is recovery. Recovery adds \(i\) to
the immune history set. Complete memory reset at rate \(\eta\) occurs only
while uninfected. There is no mutation and no simultaneous coinfection. The
equations are
\begin{align}
\dot U_A={}&d\mathbf1_{A=\varnothing}
+\eta\mathbf1_{A=\varnothing}\sum_{B\ne\varnothing}U_B
-\eta\mathbf1_{A\ne\varnothing}U_A-dU_A\notag\\
&-\sum_{i\notin A}c_{i,A}U_A\Lambda_i
+\sum_{i\in A}\beta_{i,A\setminus\{i\}}X_{A\setminus\{i\},i},\label{eq:UA}\\
\dot X_{A,i}={}&c_{i,A}U_A\Lambda_i-(\beta_{i,A}+d)X_{A,i},
\qquad i\notin A.\label{eq:XAi}
\end{align}
Let \(x_{20}\) collect all 20 coordinates \(U_A\) and \(X_{A,i}\). As in
Proposition~\ref{prop:simplex8}, summation gives
\(\frac{\dd}{\dd t}\one^\top x_{20}=d(1-\one^\top x_{20})\), so the
19-dimensional probability simplex is positively invariant.

Let \(\cF_{12}\) be the strain-3-free invariant face on which every active strain-3
infection and every state whose history contains strain 3 vanishes. Its
nonzero states are
\[
(U_\varnothing,X_{\varnothing,1},X_{\varnothing,2},
X_{\{2\},1},X_{\{1\},2},U_{\{1\}},U_{\{2\}},U_{\{1,2\}}),
\]
which identify with the eight states of Section~\ref{sec:resident}. On this
face we set
\[
\begin{gathered}
c_{1,\varnothing}=\alpha_1,\quad c_{1,\{2\}}=\gamma_1,\quad
c_{2,\varnothing}=\alpha_2,\quad c_{2,\{1\}}=\gamma_2,\\
\tau_{1,\varnothing}=\tau_{1,\{2\}}=
\tau_{2,\varnothing}=\tau_{2,\{1\}}=1,\\
\beta_{1,\varnothing}=\beta_1,\quad
\beta_{1,\{2\}}=\beta_{1\mid2},\quad
\beta_{2,\varnothing}=\beta_2,\quad
\beta_{2,\{1\}}=\beta_{2\mid1}.
\end{gathered}
\]
These parameter identifications make the restriction exact.

In plain terms, setting every strain-3 state to zero leaves exactly the
two-strain model, rather than an approximation to it. We therefore treat this
invariant two-strain boundary as the resident subsystem into which strain 3 is
introduced.

\subsection{Linearized invasion dynamics transverse to the resident boundary}

Here ``transverse'' means perturbing the strain-3-free resident boundary by
adding an infinitesimal amount of strain 3. It does not refer to perturbations
within the two-strain resident subsystem.

Along a resident trajectory in \(\cF_{12}\), define
\[
u(t)=(U_\varnothing,U_{\{1\}},U_{\{2\}},U_{\{1,2\}})^\top
\]
and order the four rare active infections as
\[
z_3=(X_{\varnothing,3},X_{\{1\},3},X_{\{2\},3},X_{\{1,2\},3})^\top.
\]
Let \(h_3\) contain the remaining states whose histories include strain 3,
and let \(y\) denote resident variables. Put
\(\nu_{3,A}=\beta_{3,A}+d\). The strain-3 transverse coefficient
matrix is
\begin{equation}
L_3(t)=\bigl(u(t)\odot c_3\bigr)\tau_3^\top-\operatorname{diag}(\nu_3).
\label{eq:L3}
\end{equation}

\begin{proposition}[Exact face and transverse reduction]
\label{prop:block}
The face \(\cF_{12}\) is positively invariant, and its restriction is exactly
the system from \eqref{eq:residentS} through \eqref{eq:residentP}. Ordered as
\((z_3,h_3,y)\), the
coefficient matrix of the full variational equation along a resident solution
is lower block triangular,
\begin{equation}
J(t)=
\begin{pmatrix}
L_3(t)&0&0\\
C(t)&H_3(t)&0\\
B_z(t)&B_h(t)&J_{12}(t)
\end{pmatrix}.
\label{eq:blockJ}
\end{equation}
The homogeneous subsystem of states with a history of strain 3
\(\dot h_3=H_3(t)h_3\) is uniformly exponentially stable; for every
nonnegative solution of this homogeneous subsystem,
\[
\|h_3(t)\|_1\le \e^{-dt}\|h_3(0)\|_1.
\]
Consequently, the only new transverse invasion instability is determined by
\(L_3\).
\end{proposition}
\begin{proof}
With no mutation into strain 3, all strain-3 coordinates have zero inflow on
\(\cF_{12}\). The surviving fluxes match the resident equations term by term.
At first order, \(z_3\) produces itself and can feed states that record a
previous strain-3 infection; those states cannot produce homologous strain-3
infection. This gives \eqref{eq:blockJ}. Infection and recovery by strains 1
and 2 only transfer mass within this history block. Summing its
homogeneous equations leaves background
death and, for uninfected states whose recorded history includes strain 3,
additional memory reset, so
\(\frac{\dd}{\dd t}\one^\top h_3\le-d\one^\top h_3\). Because this linear
subsystem preserves nonnegativity, the stated \(\ell^1\)-norm bound follows.
\end{proof}

Thus the initial growth of strain 3 is governed entirely by four active
strain-3 infection states. The remaining new states carry a history of strain
3 after recovery; they receive first-order input from active infection but
cannot feed back into it because homologous reinfection is excluded.

The trivial Floquet multiplier associated with time translation belongs to
\(J_{12}\), not to \(L_3\). None of the four transverse invasion multipliers
may therefore
be discarded as a phase mode.

\subsection{Equilibrium and periodic invasion thresholds}

At an equilibrium, the environment experienced by a rare strain is constant,
so a next-generation matrix compares new infections with removals. Its
spectral radius is the invasion reproduction number: values above one mean
initial growth, whereas values below one mean initial decline. Along a
resident cycle, that environment changes over time. The appropriate analogue
is the one-period linear map, or monodromy matrix: its principal Floquet
multiplier governs the asymptotic growth per period of positive infinitesimal
strain-3 perturbations.

\begin{proposition}[Invasion criteria at equilibrium and periodic resident attractors]
\label{prop:criteria}
Let \(a=u(E)\odot c_3\) at a resident equilibrium \(E\in\cF_{12}\). Then the
rank-one next-generation matrix and strain-3 invasion reproduction number at
\(E\) are
\[
\cK_E=a(\tau_3/\nu_3)^\top,\qquad
\mathcal R_{3\mid E}=(\tau_3/\nu_3)^\top a,
\]
and the spectral bound \(r_3(E)=\sbound(L_3(E))\) has the sign of
\(\mathcal R_{3\mid E}-1\).

If \(\Gamma\subset\cF_{12}\) is a hyperbolic orbitally asymptotically stable
resident cycle of period \(T\), with every nontrivial resident Floquet
multiplier strictly inside the unit circle, let
\[
\dot\Phi_3=L_3(t)\Phi_3,\quad \Phi_3(0)=I_4,\qquad \cM_3=\Phi_3(T).
\]
The periodic invasion multiplier and exponent are
\[
\mu_3(\Gamma)=\rho(\cM_3),\qquad
r_3(\Gamma)=T^{-1}\log\mu_3(\Gamma).
\]
The resident cycle is linearly stable with respect to strain-3 perturbations
when \(r_3(\Gamma)<0\) and linearly unstable with respect to them when
\(r_3(\Gamma)>0\). Equivalently, the principal invasion multiplier is below
or above one, respectively.
\end{proposition}
\begin{proof}
At equilibrium, the infection operator is the rank-one matrix
\(a\tau_3^\top\), and the removal operator is positive diagonal. The standard
next-generation threshold gives the stated sign relation. Along a
periodic resident orbit, Floquet theory applied to the positive strain-3
invasion block
gives the second statement. Proposition~\ref{prop:block} excludes the stable
history block from changing this threshold.
\end{proof}

For the numerical examples, infectiousness and total removal do not depend on
pre-infection history:
\[
\tau_{3,A}=\tau_0,\qquad \nu_{3,A}=\nu_0.
\]
This special case exposes the mechanism without eliminating the general
Floquet formulation: the invasion exponent depends only on a weighted sum of
the resident immune history fractions averaged over time.

For an ergodic invariant probability measure \(\mathfrak m\) of the resident
subsystem, define
\(r_3(\mathfrak m)\) as the top Lyapunov exponent of the strain-3 invasion
block along
\(\mathfrak m\)-almost every resident trajectory; it exists under the bounded
coefficients here by the multiplicative ergodic theorem. An equilibrium gives
a Dirac invariant measure, and a periodic orbit gives its normalized
occupation measure.

\begin{corollary}[Reduction for invariant measures]
\label{cor:measure}
For an ergodic invariant probability measure \(\mathfrak m\) of the resident
subsystem, let
\(\bar u_{\mathfrak m}=\int u\,\dd\mathfrak m\). Suppose strain-3 infectiousness
and total removal are independent of pre-infection history across the four
active infection classes. Then
\begin{equation}
r_3(\mathfrak m)=\tau_0c_3^\top\bar u_{\mathfrak m}-\nu_0.
\label{eq:measure_r}
\end{equation}
For a period-\(T\) orbit, its principal Floquet multiplier, the Perron root
of \(\cM_3\), is \(\exp\{T r_3(\Gamma)\}\).
\end{corollary}
\begin{proof}
Writing \(Z_3=\one^\top z_3\) and summing the strain-3 invasion equations gives the scalar
linear equation
\[
\dot Z_3=\{\tau_0c_3^\top u(t)-\nu_0\}Z_3.
\]
Moreover, because
\(L_3(t)=\tau_0(u(t)\odot c_3)\one^\top-\nu_0I\), the subspace
\(\ker(\one^\top)\) is invariant and has exponent \(-\nu_0\). The scalar
exponent obtained from \(Z_3\) is
\(\tau_0c_3^\top\bar u_{\mathfrak m}-\nu_0\ge-\nu_0\), so it is the top
Lyapunov exponent of the strain-3 invasion block. Time averaging proves
\eqref{eq:measure_r}; one-period integration gives the principal Floquet
multiplier formula.
\end{proof}

For \(\alpha_3>0\), \(\sigma_{31}>0\), \(\sigma_{32}>0\), and
\(\theta_{312}\in\R\), we parameterize the invader's history-specific
effective transmission coefficients by
\begin{equation}
c_3=\alpha_3\bm w_3(\sigma_{31},\sigma_{32},\theta_{312}),\qquad
\bm w_3=(1,\sigma_{31},\sigma_{32},
\sigma_{31}\sigma_{32}\e^{\theta_{312}})^\top.
\label{eq:invader_parameters}
\end{equation}
The examples set \(\theta_{312}=0\). The effects of histories 1 and 2
therefore combine multiplicatively; the benchmark contains no extra interaction term for
having both prior infections.

\begin{corollary}[Open regions of invader parameter space]
\label{cor:open}
Assume \(\tau_0>0\) and \(\nu_0>0\). Fix
\((\sigma_{31},\sigma_{32},\theta_{312})\) and write
\[
D_E=\bm w_3^\top u(E),\qquad
D_\Gamma=\bm w_3^\top\overline{u(\Gamma)},\qquad
\alpha_A^*=\frac{\nu_0}{\tau_0D_A}.
\]
If \(D_\Gamma>D_E>0\), then
\[
\alpha_\Gamma^*<\alpha_3<\alpha_E^*
\quad\Longrightarrow\quad r_3(E)<0<r_3(\Gamma).
\]
If \(D_E>D_\Gamma>0\), the reversed interval gives
\(r_3(\Gamma)<0<r_3(E)\). Every strict sign ordering persists in an open
neighborhood of the invader parameter vector.
\end{corollary}
\begin{proof}
Substitution of \eqref{eq:invader_parameters} into \eqref{eq:measure_r} makes each exponent
affine in the positive scale \(\alpha_3\). The threshold ordering follows by
inverting \(D_E\) and \(D_\Gamma\). Continuity in all invader parameters
preserves strict inequalities locally.
\end{proof}

\section{Numerical methods and verification}
\label{sec:methods}

All stability calculations respect the simplex constraint. Coexistence
equilibria are corrected in seven tangent coordinates, and the augmented Hopf
system solves simultaneously for the equilibrium, critical eigenvectors,
frequency, and parameters. Lyapunov coefficients use the factorial convention
stated in Section~\ref{sec:resident}; the second coefficient is obtained from
the fifth-order homological equations. Pseudo-arclength continuation follows
the Hopf locus through a turning point in its projection onto the parameter plane.

Periodic orbits are represented by multiple shooting on the seven-dimensional
tangent space with a frozen phase condition. The LPC equation uses
a minimally augmented singularity test for the fixed-parameter shooting
Jacobian. Reported points use 16 shooting segments and are checked at selected
parameters with 32 segments and with both DOP853 and Radau.
Resident Floquet multipliers are computed from the cyclic generalized
eigenvalue problem formed by the segment variational maps and are cross-checked
against a monodromy product projected onto a Poincar\'e section; the direction normal to the
simplex is excluded. The LPC curve is also compared with independent
one-parameter continuation calculations at fixed parameter values.

Strain-3 invasion on the resident cycle is evaluated by integrating the
resident orbit and the \(4\times4\) invasion variational equation jointly, avoiding
interpolation of a stored cycle. The principal multiplier is cross-checked
across solver, tolerance, and initial phase. The exact formula for the
history-independent case in
Corollary~\ref{cor:measure} supplies an independent exact reduction. Finally,
finite inoculations transfer an inoculum \(\varepsilon\) from \(S\) to
\(X_{\varnothing,3}\) in the full 20-state system. One-period endpoint growth
is compared with the prediction from the strain-3 invasion block for three inoculum sizes, several
cycle phases, and two solvers.

The computations provide reproducible floating-point evidence rather than
proofs based on interval arithmetic. Each calculation is accompanied by numerical
tolerances, source and input hashes, restart or trajectory data, and a file
manifest. We report numerical results only when checks of residuals, conditioning,
agreement between solvers, nonnegativity, and mass conservation pass. The
supplement gives the full 20-state parameter map, algorithms, tolerances,
worst-case diagnostic values, and file locations needed for reproduction.

Unless stated otherwise, numerical values in the main text are rounded to
about six significant digits. All calculations use the full-precision values
stored in the archived result files.

\section{Resident bistability organized by Bautin and LPC bifurcations}
\label{sec:resident_results}

At the parameter point used for the invasion experiment,
\[
p_0=(\gamma_2,\eta)=(1,10^{-3}),
\]
the resident strains can approach either a stable endemic equilibrium or a
stable recurring outbreak cycle. The spectral abscissa of the Jacobian
restricted to the simplex tangent space, \(J_\D\), is approximately \(-1.346\times10^{-3}\). The stable
cycle has period \(T\approx510.05\), and the maximum modulus of its
nontrivial resident Floquet multipliers is \(0.4834\), well
inside the unit circle.

Figure~\ref{fig:resident}a shows their bistability on the
\(\eta=10^{-3}\) slice. The equilibrium loses stability at the lower Hopf point
\(H_1\), regains it at the upper Hopf point \(H_2\), and is stable at \(p_0\).
The supercritical lower Hopf produces a branch of stable periodic orbits. The
subcritical upper Hopf produces a branch of unstable small-amplitude periodic
orbits. That unstable branch meets the stable large-amplitude branch at the
LPC. Consequently, for
\(H_2<\gamma_2<\gamma_{2,F}\), including \(p_0\), the stable equilibrium and
stable large-amplitude cycle are both stable, separated locally by the
unstable cycle. Panel (b)
shows that these are dynamically distinct host environments: the equilibrium
has constant prevalence, whereas the cycle contains alternating peaks of the
two strains.

The one-parameter picture is part of a two-parameter structure. In
Figure~\ref{fig:resident}c the two Hopf points belong to one regular Hopf arc in
\((\gamma_2,\eta)\). The LPC curve approaches that arc where the first
Lyapunov coefficient changes sign. This is the generalized Hopf, or Bautin,
point. The shaded region is the candidate bistability region bounded by the
computed upper Hopf and LPC branches; \(p_0\), where both resident
attractors are verified directly, is marked explicitly. The local unfolding
of the Bautin point gives rise to the nearby Hopf and LPC branches and the
associated resident attractors.
It does not itself set the strain-3 invasion sign.

\begin{computationalresult}[Computed resident bifurcation structure]
\label{result:bautin}
Numerical solution of the augmented system locates the generalized Hopf point at
\[
(\gamma_{2,B},\eta_B)\approx(0.576793,\,0.00309825),
\]
with frequency \(\omega_B\approx0.0256644\), \(l_1\) numerically zero, and
\(l_2\approx-399.75\). All other tangent-space eigenvalues have negative real parts,
and the Hopf crossing and \(l_1\)-variation are nondegenerate.

A computed sequence of 22 numerically nondegenerate, finite-amplitude LPC
points spans
\(10^{-3}\le\eta\le0.0030925\), from
\[
(\gamma_{2,F},T_F)\approx(1.21018,470.285)
\]
to
\[
(\gamma_{2,F},T_F)\approx(0.577418,245.081).
\]
Its amplitude, period, parameter separation between the LPC and Hopf curves, and approach
direction in parameter space are
quantitatively consistent with the Bautin normal form coefficients. The last
orbit has finite amplitude \(9.79\times10^{-3}\); the computation does not
represent the Bautin point itself as a periodic boundary-value solution.
\end{computationalresult}

\begin{figure}[p]
\centering
\includegraphics[width=\textwidth]{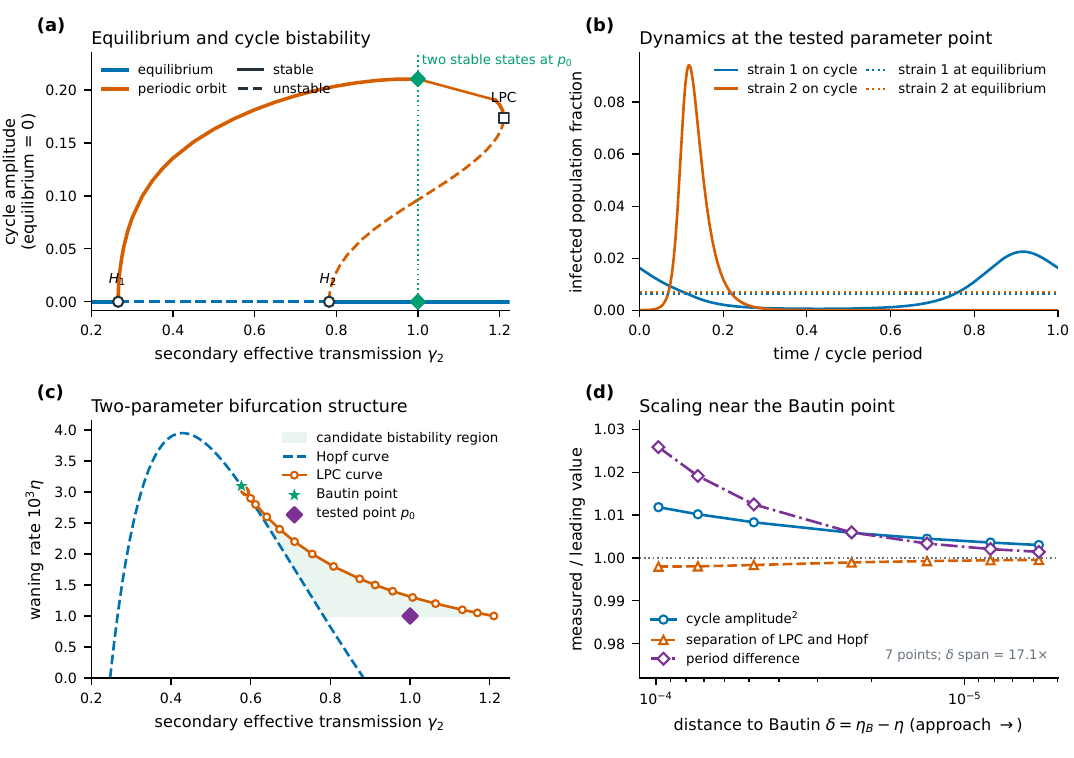}
\caption{Resident dynamics and the bifurcations that organize them. (a) Branch
diagram at \(\eta=10^{-3}\). The equilibrium is drawn at zero cycle amplitude;
solid and dashed segments denote stable and unstable equilibria or periodic
orbits. The two green
diamonds at \(p_0\) mark the stable equilibrium and stable large-amplitude
cycle. \(H_1\) and \(H_2\) are Hopf points; LPC denotes a limit point of
cycles, equivalently a fold of periodic orbits.
(b) Total prevalence of strains 1 and 2 over the stable cycle at \(p_0\);
dotted lines give their constant equilibrium values. (c) Hopf and LPC
curves in the two-parameter plane. The lightly shaded candidate bistability
region is bounded by the computed upper Hopf and LPC branches; stability of both attractors is verified
at \(p_0\), not throughout every interior parameter value. The plotted LPC
sequence stops at its last computed finite-amplitude orbit and is not
extrapolated to the Bautin point. (d) For the seven LPC points closest to
Bautin, three measured quantities are divided by their normal form leading
terms. Their approach to one is consistent with the local Bautin scaling over
the resolved window.}
\label{fig:resident}
\end{figure}

Figure~\ref{fig:resident}d makes the local connection quantitative. The seven
computed LPC points closest to Bautin span a factor 17.1 in
\(\delta=\eta_B-\eta\). The observed squared cycle amplitude,
parameter separation between the LPC and Hopf curves, and period difference approach their normal form leading
values together. Because the nested endpoint windows do not each span a full
decade in \(\delta\), the available computation supports the local scaling
over a finite window; the branch was not continued to zero amplitude. The
separation of the one-dimensional shooting Jacobian nullspace and the distances
of the remaining Floquet multipliers from thresholds for secondary bifurcations are
shown in Supplementary Figure~S1.

Table~\ref{tab:bifurcation} collects the principal branch, stability, and
normal form diagnostics used in this resident calculation.

\begin{table}[h]
\centering
\caption{Key resident bifurcation and stability quantities. Values are rounded
floating-point results; \(l_1\) uses the factorial convention in the text.}
\label{tab:bifurcation}
\begin{tabular}{lrr}
\toprule
Object & Parameter/value & Diagnostic\\
\midrule
lower Hopf, \(\eta=10^{-3}\) & \(\gamma_2\approx0.265816\) & \(l_1\approx-43.71\)\\
upper Hopf, \(\eta=10^{-3}\) & \(\gamma_2\approx0.782566\) & \(l_1\approx35.17\)\\
Bautin & \((\gamma_2,\eta)\approx(0.576793,0.00309825)\) & \(l_2\approx-399.75\)\\
LPC, \(\eta=10^{-3}\) & \(\gamma_2\approx1.21018\) & \(T\approx470.285\)\\
resident equilibrium at \(p_0\) & N/A & \(s(J_\D)\approx-1.346\times10^{-3}\)\\
resident cycle at \(p_0\) & \(T\approx510.05\) & \(\rho_{\rm nt}\approx0.4834\)\\
\bottomrule
\end{tabular}
\end{table}
\section{Invasion by a third strain depends on the resident attractor}
\label{sec:invasion_results}

The stable equilibrium and stable cycle at \(p_0\) contain the same resident
strains \(\{1,2\}\) but leave different fractions of uninfected hosts in each
immune history class. In the order
\((S,R_1,R_2,P_{12})\), the equilibrium distribution is
\[
(0.29495,0.07061,0.11616,0.50512),
\]
whereas the time average over the cycle is
\[
(0.30194,0.06991,0.09518,0.51972).
\]
The four entries do not sum to one because currently infected residents are
not included. Figure~\ref{fig:invasion}a compares the equilibrium values with
the cycle means; its error bars show the full range over one
cycle, not confidence intervals. Although the means appear close, their
weighted difference is enough to change the sign of strain-3 growth. In this
benchmark with history-independent infectiousness and total removal,
Corollary~\ref{cor:measure} shows that only the mean host composition determines
the exact invasion exponent. The order of fluctuations within the cycle does
not affect it in this special case.

For strain 3, we set \(\tau_0=1\) and
\(\beta_{3,A}=0.1\), so \(\nu_0=0.1001\). We further set
\(\sigma_{31}=0.2\) and \(\theta_{312}=0\), and vary
\((\alpha_3,\sigma_{32})\). The two threshold curves \(r_3(E)=0\) and
\(r_3(\Gamma)=0\) exchange order at \(\sigma_{32}\approx0.379496\)
(Figure~\ref{fig:invasion}b). Here \(\sigma_{32}\) is the effective transmission
coefficient after strain 2 divided by the corresponding coefficient in a host
with no recorded history. For
\(\alpha_3\) between the two threshold curves, the two sides of their crossing
give opposite signs of invasion growth for the two attractors.

\begin{computationalresult}[Opposite invasion signs across resident attractors]
\label{result:invasion}
At the same resident parameter point \(p_0\):
\begin{enumerate}[label=(\roman*)]
\item For fixed invader parameter set A, strain 3 grows only near the cycle:
\(\sigma_{31}=0.2\), \(\sigma_{32}=0.1\), and \(\alpha_3\approx0.300338\),
\[
r_3(E)\approx-7.51957\times10^{-4}<0<r_3(\Gamma)\approx7.63427\times10^{-4},
\]
and the periodic principal multiplier is approximately \(1.4761\).
\item For fixed invader parameter set B, strain 3 grows only near the equilibrium:
\(\sigma_{31}=0.2\), \(\sigma_{32}=1\), and \(\alpha_3\approx0.192283\),
\[
r_3(\Gamma)\approx-1.06535\times10^{-3}<0<r_3(E)\approx1.08852\times10^{-3},
\]
and the periodic principal multiplier is approximately \(0.5808\).
\end{enumerate}
Within each comparison, the resident parameter vector and invader parameter set are
fixed; only the resident attractor changes.
\end{computationalresult}

Table~\ref{tab:invasion} reports the two fixed invader parameter sets and their
equilibrium and periodic invasion diagnostics. No mutation process is modeled.

\begin{table}[t]
\centering
\caption{Selected parameter sets for the third strain and invasion diagnostics at \(p_0\).
The two rows are different invader parameter sets; each row compares that same invader
against the two resident attractors.}
\label{tab:invasion}
\begin{tabular}{lrrrrr}
\toprule
Parameter set & \(\sigma_{32}\) & \(\alpha_3\) & \(r_3(E)\) & \(r_3(\Gamma)\) & \(\rho(\cM_3)\)\\
\midrule
set A: cycle growth & 0.1 & 0.300338 & \(-7.52\!\times\!10^{-4}\) & \(7.63\!\times\!10^{-4}\) & 1.4761\\
set B: equilibrium growth & 1.0 & 0.192283 & \(1.09\!\times\!10^{-3}\) & \(-1.07\!\times\!10^{-3}\) & 0.5808\\
\bottomrule
\end{tabular}
\end{table}

These are not isolated points. Both exponents are strictly increasing in
positive \(\alpha_3\) and \(\sigma_{32}\). The width of the interval between
the two thresholds is
\[
\left|\alpha_E^*-\alpha_\Gamma^*\right|
=\frac{\nu_0\left|D_\Gamma-D_E\right|}{\tau_0D_ED_\Gamma},
\]
which makes explicit how a difference in the weighted host environment creates
the region in Figure~\ref{fig:invasion}b where the sign depends on the attractor. For parameter
set A, the
rectangle
\[
\alpha_3\in[0.298837,0.301840],\qquad
\sigma_{32}\in[0.099,0.101]
\]
has \(r_3(E)<0<r_3(\Gamma)\) at the relevant corners; the minimum absolute
one-period log growth among these corner inequalities is \(0.0967\). For
parameter set B,
\[
\alpha_3\in[0.191322,0.193244],\qquad
\sigma_{32}\in[0.99,1.01]
\]
has the reversed inequalities; the corresponding minimum absolute value is
\(0.0852\). Analytic
monotonicity combined with the numerical corner values verifies the sign
pattern throughout each rectangle in floating-point arithmetic. The interiors
are the open regions asserted in Corollary~\ref{cor:open}; this remains an
slice through invader parameter space at fixed resident parameters, not a global parameter
classification.

\begin{figure}[p]
\centering
\includegraphics[width=\textwidth]{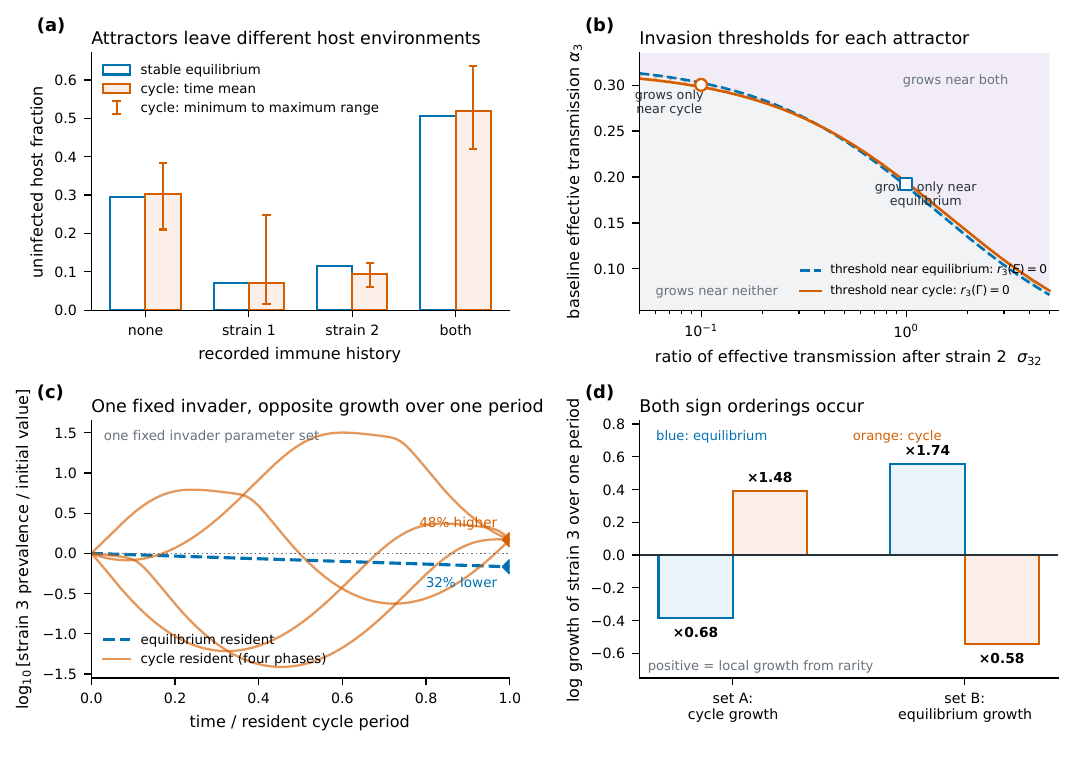}
\caption{Invasion from rarity depends on the attractor at
\((\gamma_2,\eta)=(1,10^{-3})\). (a) Fractions in the four uninfected
immune history classes available to strain 3. Blue bars are equilibrium
values; orange bars and whiskers are the cycle averages and full
range from minimum to maximum over the cycle. (b) Curves on which the invasion exponent of strain 3 is zero near the
equilibrium and cycle. Below both curves strain 3 declines near both
attractors; above both it grows near both. Between the curves, the growth sign
depends on the attractor. The circle and square mark parameter sets A and B. (c) Full
20-state simulations for the same fixed parameter set A. Near the equilibrium the
inoculum is \(32\%\) lower after one resident period. Four different starting
phases on the same resident cycle give different transient paths but the same
positive endpoint factor, \(1.48\). (d) One-period log growth for both fixed
parameter sets. Each pair holds the invader fixed while changing only the resident
attractor. Positive bars mean local growth from rarity, not long-term
establishment.}
\label{fig:invasion}
\end{figure}

The comparison with the full system comprises 60 one-period inoculation runs. All
states remain nonnegative to numerical tolerance, maximum mass error is
\(4.7\times10^{-15}\), and every observed sign agrees with the transverse
prediction. For inoculum sizes \(10^{-8},10^{-7},10^{-6}\), the log-log slopes
of nonlinear error against inoculum size lie between \(0.99910\) and \(0.99997\).
Changing the resident cycle phase or replacing DOP853 by Radau produces
differences that are negligible compared with the distances of the reported
values of log growth from zero. This
convergence is the expected local consistency check: as the inoculum tends to
zero, the nonlinear system approaches the strain-3 invasion variational
dynamics. The convergence curves as the inoculum tends to zero and the solver
overlays are retained as
Supplementary Figure~S2.

\section{Discussion}
\label{sec:discussion}

Knowing which resident strains are present is not always enough to determine
whether a rare third strain grows. At \(p_0\), the same two strains support two
stable long-term states. Those states leave different weighted distributions
of uninfected host histories. Holding every resident parameter and the invader
parameter set fixed, changing only the resident attractor reverses the invading
strain's local growth sign.

The general idea that invasion growth rates should be associated with
ergodic invariant measures supported by the boundary resident dynamics is
established in population dynamics
\citep{Schreiber2000,GeritzEtAl2002,HofbauerSchreiber2022}. Invasion-graph
results often require invasion signs to be consistent across ergodic invariant
measures with the same set of resident strains. The present model gives a concrete
epidemic example in which this consistency fails: the equilibrium and the cycle are both
stable, belong to the same strain-3-free invariant face, and give opposite invasion signs for
the same absent strain. The new contribution is therefore not the formulation of invasion
growth rates in terms of invariant measures itself, nor the first use of a periodic invasion criterion. It is the
construction specific to this model that joins autonomous bistability between an
equilibrium and a periodic orbit, an exact embedding with three strains, and a
sign reversal determined by the attractor.

The bifurcation and invasion calculations play different roles. The Bautin
point and LPC curve describe the local structure of bistability between the
equilibrium and periodic orbit. The
invasion reversal is controlled by
\[
\bm w_3^\top\{\overline{u(\Gamma)}-u(E)\}.
\]
At fixed \(\sigma_{31}=0.2\), this expression is affine in \(\sigma_{32}\) and
changes sign where the two threshold curves cross. The Bautin point does not
directly cause strain 3 to grow or decline; it supplies the alternative
resident attractors whose host compositions are then compared.

The benchmark also separates an explicit interaction between two prior
infections from an effect created by system feedback. We set
\(\theta_{312}=0\), so the effects of prior strains 1 and 2 on the effective
transmission coefficient for strain 3 multiply without an extra term for
having both histories. The sign
reversal nevertheless occurs because the two resident attractors weight those
histories differently. Such an explicit interaction is therefore not required
for this result. Allowing \(\theta_{312}\ne0\) would provide a natural next
test of how it combines with the host environment created by the resident
attractor.

When strain-3 infectiousness and total removal are independent of
pre-infection history across the four active infection classes, the periodic
exponent reduces exactly to a time average of the immune history distribution. The order of changes within the cycle does not affect
the exponent in this special case. The full Floquet calculation remains the
general criterion and supplies an independent numerical check. If
infectiousness or removal varies by history, however, the within-cycle order
can matter and the full time-ordered monodromy is then required; replacing it
with a simple average need not be valid.

The biological interpretation should remain modest. The benchmark deliberately
uses strong, phenomenological history-associated susceptibility enhancement
for strain 2 and complete
memory reset. These assumptions are useful for resolving the mechanism, but
they are not estimates for a named pathogen. An empirically calibrated extension would
need empirically estimated history-dependent transmission coefficients,
plausible waning rates, and
observation models that distinguish prevalence from host immune composition.

The claims are limited in four further ways. First, the Bautin and LPC results
are floating-point computations checked by independent numerical methods, not
interval-certified bifurcation theorems. The finite-amplitude LPC sequence is
consistent with the LPC branch approaching the generalized Hopf point with the
predicted normal form scaling, but the available nested endpoint windows do
not each span a full decade. Second, the resident
bistability is verified at the tested parameter point; we do not claim a
global basin decomposition. Third, a positive invasion exponent makes a resident
boundary attractor unstable to strain 3, but it does not identify the state
reached after strain 3 is no longer rare. Initial rare growth need not determine
subsequent frequency dynamics because the immune environment can be reshaped
during invasion \citep{BarratCharlaixNeher2024}; stable
three-strain coexistence, persistence, exclusion, and post-invasion cycles all
require a separate interior analysis. This distinction also matches modern
coexistence theory, where a positive boundary invasion growth rate is a local
diagnostic and permanence requires conditions across the relevant boundary
invariant measures or, under the assumptions of the theory of invasion graphs, across
the invasion graph \citep{HofbauerSchreiber2022}. Finally, examining the
Floquet multipliers at the 22 computed LPC points does not exclude secondary
bifurcations or chaos
elsewhere in parameter space.

These limitations suggest a focused next program: continue the two
attractor-specific invasion thresholds over the computed candidate
bistability region,
follow the dynamics of the full system after transverse loss of stability, and then
introduce independently parameterized interactions between prior infections. An extension to more than three strains
can build on the same formulation of invasion in terms of invariant measures and associated growth
rates, but the current paper deliberately resolves the mechanism with three strains first.

\FloatBarrier
\section{Conclusion}

At one parameter vector, a two-strain model with immune history supports both a
stable endemic equilibrium and a stable endemic cycle. A numerically resolved
Bautin point and LPC curve underlie this resident bistability. In the
exact three-strain embedding, the two resident attractors expose the same fixed
invader to different environments shaped by immune history and can give it opposite
local growth signs. The sign patterns persist over open regions of invader
parameter space
and are recovered by the full nonlinear model as the inoculum decreases.

The invasion threshold must therefore be defined relative to a particular
resident attractor, not only to the strains that are present. For the
equilibrium, the associated invariant measure records a fixed host
composition; for the cycle, it records how much time the system spends in
each part of the recurring trajectory. These two measures yield different
distributions of immune histories for rare strain 3 and, in our examples,
local growth rates with opposite signs.

\section*{Data and code availability}
Source code, numerical configurations, raw restart and trajectory data,
manifests, and figure inputs supporting this study are available from the
authors upon reasonable request.

\end{document}